\documentclass[11pt, a4paper]{article}
\usepackage{}
\usepackage{amsthm}
\usepackage{amsmath,amssymb,latexsym,color}
\usepackage[mathscr]{eucal}
\usepackage[colorlinks,
            linkcolor=blue,
            anchorcolor=green,
            citecolor=magenta
           ]{hyperref}
\usepackage{graphicx}
\usepackage{tikz}
\usetikzlibrary{calc}
\usepackage{CJK}

\usepackage{authblk}

\long\def\delete#1{}

\usepackage{color}

\definecolor{Blue}{rgb}{0,0,1}
\definecolor{Red}{rgb}{1,0,0}
\definecolor{DarkGreen}{rgb}{0,0.6,0}
\definecolor{DarkYellow}{rgb}{1,1,0.2}
\definecolor{DarkPurple}{rgb}{.6,0,1}

\usepackage{xcolor}
\usepackage[normalem]{ulem}

\usepackage{cleveref}
\crefformat{section}{\S#2#1#3}
\crefformat{subsection}{\S#2#1#3}
\crefformat{subsubsection}{\S#2#1#3}
\crefrangeformat{section}{\S\S#3#1#4 to~#5#2#6}
\crefmultiformat{section}{\S\S#2#1#3}{ and~#2#1#3}{, #2#1#3}{ and~#2#1#3}

\begin{document}
\setcounter{page}{1}
\newtheorem{thm}{Theorem}[section]
\newtheorem{fthm}[thm]{Fundamental Theorem}
\newtheorem{dfn}[thm]{Definition}
\newtheorem*{rem}{Remark}
\newtheorem{lem}[thm]{Lemma}
\newtheorem{cor}[thm]{Corollary}
\newtheorem{exa}[thm]{Example}
\newtheorem{prop}[thm]{Proposition}
\newtheorem{prob}[thm]{Problem}
\newtheorem{fact}[section]{Fact}
\newtheorem{con}[thm]{Conjecture}
\renewcommand{\thefootnote}{}
\newcommand{\remark}{\vspace{2ex}\noindent{\bf Remark.\quad}}
\newtheorem{ob}[thm]{Observation}
\newcommand{\rmnum}[1]{\romannumeral #1}
\renewcommand{\abovewithdelims}[2]{%
\genfrac{[}{]}{0pt}{}{#1}{#2}}

\newcommand\Sy{\mathrm{S}}
\newcommand\Cay{\mathrm{Cay}}
\newcommand\ex{\mathrm{ex}}
\newcommand\supp{\mathrm{supp}}


\def\qed{\hfill$\Box$\vspace{11pt}}

\title {\bf  Extremal even-cycle-free subgraphs of the complete transposition graphs}

\author[1,2]{Mengyu Cao\thanks{E-mail: \texttt{caomengyu@mail.bnu.edu.cn}}}
\author[2]{Benjian Lv\thanks{Corresponding author. E-mail: \texttt{bjlv@bnu.edu.cn}}}
\author[2]{Kaishun Wang\thanks{E-mail: \texttt{wangks@bnu.edu.cn}}}
\author[3]{Sanming Zhou\thanks{E-mail: \texttt{sanming@unimelb.edu.au}}}
\affil[1]{\small Department of Mathematical Sciences, Tsinghua University, Beijing 100084, China}
\affil[2]{\small Laboratory of Mathematics and Complex Systems (Ministry of Education), School of Mathematical Sciences, Beijing Normal University, Beijing 100875, China}
\affil[3]{\small School of Mathematics and Statistics, The University of Melbourne, Parkville, VIC 3010, Australia}

\date{}

\openup 0.5\jot
\maketitle

\begin{abstract}
Given graphs $G$ and $H$, the generalized Tur\'{a}n number ${\rm ex}(G,H)$ is the maximum number of edges in an $H$-free subgraph of $G$. In this paper, we obtain an asymptotic upper bound on $\ex(CT_n,C_{2l})$ for any $n \ge 3$ and $l\geq2$, where $C_{2l}$ is the cycle of length $2l$ and $CT_n$ is the complete transposition graph which is defined as the Cayley graph on the symmetric group $\Sy_n$ with respect to the set of all transpositions of $\Sy_n$.

\vspace{2mm}

\noindent{\bf Key words}\ \ Tur\'{a}n number,  even-cycle-free subgraph,  complete transposition graph,  Ramsey-type problem

\

\noindent{\bf MSC2010:} \   05C35, 05C38, 05D05


\end{abstract}

\section{Introduction}

Throughout this paper graphs are finite and undirected with no loops or multiple edges. The vertex and edge sets of a graph $G$ are denoted by $V(G)$ and $E(G)$, respectively. The numbers of vertices and edges of $G$ are denoted by $v(G)$ and $e(G)$, respectively. The degree of a vertex $x \in V(G)$ in $G$ is denoted by $d_{G}(x)$, and the edge joining vertices $u$ and $w$ are denoted as an unordered pair $\{u, w\}$. A cycle with $l$ edges is called an \emph{$l$-cycle} or a \emph{cycle of length $l$}, where $l \ge 3$. A path with length $l$ is called an \emph{$l$-path}, where $l \ge 1$. As usual an $l$-cycle is denoted by $C_l$ and an $l$-path by $P_l$. Two graphs $G$ and $H$ are said to be \emph{isomorphic} if there exists a bijection $f$ from $V(G)$ to $V(H)$ such that $\{x,y\}\in E(G)$ if and only if $\{f(x),f(y)\}\in E(H)$.

Let $G$ and $H$ be graphs. We say that $G$ is \emph{$H$-free} if there exists no subgraph of $G$ which is isomorphic to $H$. The \emph{generalized Tur\'{a}n number} ${\rm ex}(G,H)$ is the maximum number of edges in an $H$-free spanning subgraph of $G$. This invariant proposed by Erd\H{o}s \cite{R5} is a generalization of the well-known Tur\'{a}n number ${\rm ex}(n,H)$ which gives the maximum number of edges in an $H$-free graph with $n$ vertices. In the literature there is a huge amount of work on Tur\'{a}n numbers and generalized Tur\'{a}n numbers, beginning with Mantel \cite{Mantel1907} who proved that ${\rm ex}(n, K_3) = \lfloor n^2/4 \rfloor$ and Tur\'{a}n \cite{R9} who determined ${\rm ex}(n, K_r)$ for any $r \ge 3$, where $K_r$ is the complete graph with $r$ vertices. In \cite{Erdos1982}, Erd\H{o}s and Simonovits obtained an asymptotic formula for ${\rm ex}(n,H)$ in terms of the chromatic number of $H$. But when $H$ is bipartite the situation is considerably more complicated, and we can only deduce that ${\rm ex}(n,H)=o(n^2)$. Herein and in the rest of this paper asymptotics are taken as $n \rightarrow \infty$. In general, it is a challenging problem to determine ${\rm ex}(G,H)$ when $H$ is a bipartite graph, especially when $H$ is an even cycle. In this regard, two interesting functions that have received much attention are ${\rm ex}(G, K_{s,t})$ and ${\rm ex}(Q_n, C_{2l})$, where $K_{s,t}$ is the complete bipartite graph with $s$ and $t$ vertices, respectively, in the biparts of its bipartition, and $Q_n$ is the $n$-dimensional hypercube. The problem of determining ${\rm ex}(K_{m,n},K_{s,t})$, proposed by Zarankiewicz in \cite{R11}, is the analogue of Tur\'{a}n's original problem (the one of determining ${\rm ex}(K_n, K_r) = {\rm ex}(n, K_r)$) for bipartite graphs, and an excellent survey on this problem can be found in \cite{Furedi-simonvits}. Besides, some related research was dedicated to showing ${\rm ex}(G,K_{t,t})$, where $G$ is some other certain restricted graph. See e.g. \cite{Fox1,Fox2}.

The study of ${\rm ex}(Q_n, C_{2l})$ began with a problem raised by Erd\H{o}s which asks for the maximum number of edges in a $C_4$-free spanning subgraph of $Q_n$. In \cite{R5}, Erd\H{o}s conjectured that $\left(\frac{1}{2}+o(1)\right)e(Q_n)$ should be an upper bound for ${\rm ex}(Q_n, C_{4})$, and he also asked whether $o(e(Q_n))$ edges of $Q_n$ would ensure the existence of a cycle $C_{2l}$ for $l\geq 3$. The best known upper bound for ${\rm ex}(Q_n, C_{4})$, obtained by Balogn et al. \cite{Balogh2014} and improved slightly the bounds of Chung \cite{R2} and Wagner \cite{R10}, is $(0.6068+o(1))e(Q_n)$. The problem of determining the value of  ${\rm ex}(Q_n, C_{2l})$ when $l=3$ or $5$ is still open too, and progresses can be found in \cite{Alon2007,Alon2006,Balogh2014,R2,Conder}. For $l\geq 2$, upper bounds for ${\rm ex}(Q_n, C_{4l})$ and ${\rm ex}(Q_n,C_{4l+6})$ were obtained by Chung \cite{R2} and F\"{u}redi and \"{O}zkahya \cite{R1}, respectively, and their results together imply that ${\rm ex}(Q_n,C_{2l'})=o(e(Q_n))$ for $l'\geq6$ or $l'=4$. In \cite{R7}, Conlon proved that ${\rm ex}(Q_n, H) = o(e(Q_n))$ for any graph $H$ that admits a $k$-partite representation. This gives a unified approach to the proof that ${\rm ex}(Q_n, C_{2l}) = o(e(Q_n))$ for all $l \ne 5$ no less than $4$. The doubled Johnson graphs $J(n;k,k+1)$, where $1 \le k \le (n-1)/2$, form an interesting family of spanning subgraphs of $Q_n$, and in particular the doubled odd graph $\widetilde{O}_{k+1} := J(2k+1;k,k+1)$ is known to be distance-transitive. Recently, Cao et al. \cite{Cao2019} studied $\ex(J(n;k,k+1), C_{2l})$ and proved among other things that $\ex(\widetilde{O}_{k+1},C_{2l})=o(e(\widetilde{O}_{k+1}))$ for $l\geq 6$.

In this paper, we study the generalized Tur\'{a}n number $\ex(CT_n, C_{2l})$ for the complete transposition graph $CT_n$, where $n \ge 3$ and $l \ge 2$. The complete transposition graphs are an important family of Cayley graphs which share several interesting properties with hypercubes. For example, both $CT_n$ and $Q_n$ are bipartite and arc-transitive, with only integral eigenvalues, and both graphs are popular topologies for interconnection networks \cite{Hey}. Over the years several aspects of complete transposition graphs such as automorphisms, eigenvalues, connectivity and bisection width have been studied as one can find in, for example, \cite{Ganesan2015,Jwo1996,Kalpakis1997,Stacho1998,Wang2015}. In general, given a group $G$ with identity element $1$ and an inverse-closed subset $S$ of $G\setminus\{1\}$, the \emph{Cayley graph} $\Cay(G,S)$ on $G$ with respect to the \emph{connection set} $S$ is defined to be the graph with vertex set $G$ such that $x, y\in G$ are adjacent if and only if $yx^{-1}\in S$.  The \emph{complete transposition graph} $CT_n$ is defined as the Cayley graph on the symmetric group $\Sy_n$ whose connection set is the set of all transpositions of $\Sy_n$. That is,
$$
V(CT_n) = \Sy_n,
$$
$$
E(CT_n) = \{\{x, y\}: x, y \in \Sy_n \text{ and } y = ux \text{ for some transposition $u$ of } \Sy_n\}.
$$
It follows that $CT_n$ is a connected ${n\choose 2}$-regular bipartite graph with
$$
v := v(CT_n) = n!
$$
vertices and
$$
e(CT_n) = \frac{v}{2}{n\choose 2}
$$
edges.

The main result in this paper is as follows.

\begin{thm}\label{main theorem}
Let $n$ and $l$ be integers with $n \ge 3$ and $l\geq2$.
\begin{itemize}
\item[{\rm(i)}] If $l\geq4$ and $l$ is even, then $\ex(CT_n,C_{2l})=O(n^{-1+\frac{2}{l}}) e(CT_n).$

\item[{\rm(ii)}] If $l\geq4$ and $l$ is odd, then
$$
{\ex}(CT_n, C_{2l}) = \begin{cases}
O(n^{-\frac{1}{l}})e(CT_n), & \mbox{if}\ l=7,\\
O(n^{-\frac{1}{8}+\frac{1}{4(l-3)}})e(CT_n), & \mbox{otherwise}.
\end{cases}
$$

\item[{\rm(iii)}] If $l=3$, then $\ex(CT_n,C_{2l})\leq (\sqrt{2}-1+o(1))e(CT_n).$

\item[{\rm(iv)}] If $l=2$, then $\ex(CT_n,C_{2l})\leq \frac{3}{4}e(CT_n)$.
\end{itemize}
\end{thm}

An immediate consequence of Theorem \ref{main theorem} is that ${\rm ex}(CT_n, C_{2l})=o(e(CT_n))$ for $l\geq 4$. This leads to the following Ramsey-type result.

\begin{cor}\label{cor1}
Let $t$ and $l$ be integers with $t \ge 1$ and $l\geq4$. If $CT_n$ is edge-partitioned into $t$ subgraphs, then one of the subgraphs must contain $C_{2l}$ provided that $n$ is sufficiently large (depending only on $t$ and $l$).
\end{cor}

In the next section we will prove some basic properties of cycles in the complete transposition graphs. Using these preparations we will prove parts (i)-(ii) and (iii)-(iv) of Theorem \ref{main theorem} in Sections \ref{large} and \ref{small}, respectively.

\section{Preliminaries}

We assume that $\Sy_n$ is the symmetric group on $\{1,2,\ldots,n\}$, where $n \ge 3$. The identity element of $\Sy_n$ is denoted by $\mathrm{id}$. The \emph{support} of an element $x \in \Sy_n$ is defined as ${\rm supp}(x)=\{i \in\{1,2,\ldots,n\}\mid i^x \neq i\}$.

\begin{dfn}\label{direction}
{\em The \emph{support} of an edge $\{u,z\}$  of $CT_n$, denoted by $\supp(\{u,z\})$, is defined to be the support of the transposition $z u^{-1}$. That is, $\supp(\{u,z\}) = \supp(z u^{-1})$.}
\end{dfn}

Since $CT_n$ is a Cayley graph on the symmetric group $\Sy_n$ whose connection set consists of all transpositions, we know that $\supp(\{u,z\})$ is a $2$-subset of $\{1,2,\ldots,n\}$ for any $\{u,z\}\in E(CT_n)$, and $\supp(\{u,z\}) = \supp(\{z,u\})$. Note that, for any two incident edges $\{x,u\}$ and $\{x,z\}$ of $CT_n$, we have $|\supp(\{x,u\})\cap \supp(\{x,z\})| = 0$ or $1$. For any subgraph $H$ of $CT_n$, we define
$$
\supp(H):=\bigcup\limits_{\{u,z\}\in E(H)}\supp(\{u,z\}).
$$
Let $P=(u_1,u_2,\ldots,u_t)$ be a path in $CT_n$. Setting $w_i = u_{i} u_{i-1}^{-1}$ for $i\in\{2,3,\ldots,t\}$, we have $u_{t} u_1^{-1} = w_tw_{t-1}\cdots w_3w_2$ and hence $\supp(u_{t} u_1^{-1})\subseteq \supp(P).$

\begin{lem}\label{4cycle_1}
Let $g$ and $h$ be distinct transpositions of $\Sy_n$. Then the only $4$-cycles in $CT_n$ passing through the $2$-path $(g, \mathrm{id}, h)$ are $(\mathrm{id}, g, hg, h, \mathrm{id})$ and $(\mathrm{id}, g, gh, h, \mathrm{id})$. In particular, if $gh = hg$, then these $4$-cycles are identical and they are the only $4$-cycle in $CT_n$ passing through the $2$-path $(g, \mathrm{id}, h)$.
\end{lem}
\begin{proof}
Suppose $gh=hg$. Then $|\supp(g)\cap\supp(h)|=0$. Note that $\mathrm{id}, g$ and $h$ are three vertices in $CT_n$. Let $w$ be a common neighbor of the vertices $g$ and $h$ in $CT_n$. Then there exist transpositions $x,y$ such that $xg = yh = w$, implying that $gh=xy$. Since the supports of $g$ and $h$ are disjoint, the equation $gh=xy$ holds if and only if $g=x$ and $h=y$, or $g=y$ and $h=x$. Therefore, $w$ is either the vertex $\mathrm{id}$ or the vertex $gh$. Thus, there exists a unique $4$-cycle in $CT_n$ passing through $g, \mathrm{id}$ and $h$, which is $(\mathrm{id},g,h g=g h,h,\mathrm{id})$.

Suppose $g h \neq h g$. Then $|\supp(g)\cap\supp(h)|=1$. Without loss of generality we may assume $g=(1,2)$, and $h=(1,3)$. Let $w$ be a common neighbor of the vertices $g$ and $h$ in $CT_n$. Then there exist transpositions $x,y$ such that $xg=yh=w$, implying that $xy=gh=(1,2)(1,3)=(1,2,3)$. If we decompose $(1,2,3)$ into the product of two transpositions of $\Sy_n$, then the supports of these two transpositions must lie in $\{1,2,3\}$ and contain exactly one common letter. Therefore, the only ways to decompose $(1,2,3)$ into the product of two transpositions of $\Sy_n$ are $(1,2,3)=(1,3)(2,3)=(2,3)(1,2)=(1,2)(1,3)$. Hence, we have $x=(1,3)$ and $y=(2,3)$, or $x=(2,3)$ and $y=(1,2)$, or $x=(1,2)$ and $y=(1,3)$, yielding $w\in\{\mathrm{id},(1,3,2),(1,2,3)\}$. Therefore, there are exactly two $4$-cycles in $CT_n$ passing through $g, \mathrm{id}$ and $h$, namely $(\mathrm{id},g, h g,h,\mathrm{id})$ and $(\mathrm{id},g,g h,h,\mathrm{id})$.
\end{proof}

It is well known that any permutation in $\Sy_n$ can be expressed as a product of transpositions, and for each $g \in \Sy_n$ the map $\hat{g}: h \mapsto h g$, $h \in \Sy_n$ defines an automorphism of $CT_n$. Hence Lemma \ref{4cycle_1} implies the following result.

\begin{cor}\label{rem1}
Let $(u,x,z)$ be a $2$-path in $CT_n$. If $|\supp(\{x,u\})\cap \supp(\{x,z\})|=0$, then there is a unique $4$-cycle in $CT_n$ containing $(u,x,z)$, namely $(x,z,zx^{-1}u,u,x)$; and if $|\supp(\{x,u\})\cap \supp(\{x,z\})|=1$, then there are exactly two $4$-cycles in $CT_n$ containing $(u,x,z)$, namely $(x,z,zx^{-1}u,u,x)$ and $(x,z,ux^{-1}z,u,x)$.
\end{cor}

Denote by $n(C_4)$ the number of $4$-cycles in $CT_n$. Lemma \ref{4cycle_1} and Corollary \ref{rem1} together imply the following result.

\begin{cor}\label{4cycle_2}
The following hold.
\begin{itemize}
\item[{\rm(i)}] The length of a shortest cycle in $CT_n$ is $4$.
\item[{\rm(ii)}] For any edge $\{u,z\}$ of $CT_n$, there are exactly $\frac{1}{2}(n-2)(n+1)$ cycles of length $4$ in $CT_n$ containing $\{u,z\}$.
\item[{\rm(iii)}] $n(C_4) = \frac{1}{8}(n-2)(n+1)e(CT_n)$.
\end{itemize}
\end{cor}
\begin{proof}
Since $CT_n$ is bipartite, it does not contain any $3$-cycle. On the other hand, $4$-cycles exist in $CT_n$ by Lemma \ref{4cycle_1}. So any shortest cycle in $CT_n$ has length $4$ as stated in (i).

For any edge $\{u,z\}$ of $CT_n$, there are exactly ${{n-2}\choose 2}$ $2$-paths $(u,z,w)$ such that $|\supp(\{u,z\})\cap \supp(\{z,w\})|=0$, and there are exactly $n-2$ $2$-paths $(u,z,w)$ such that $|\supp(\{u,z\})\cap \supp(\{z,w\})|=1$. Hence, by Corollary \ref{rem1}, the number of $4$-cycles containing any given edge of $CT_n$ is equal to ${{n-2}\choose 2}+2(n-2)=\frac{1}{2}(n-2)(n+1)$ as claimed in (ii). We obtain (iii) from (ii) immediately.
\end{proof}

Let
$$
\mathscr{F}_0 = \{\text{all transpositions of } \Sy_n\}
$$
and
$$
\mathscr{F}_i = \{x \in \mathscr{F}_0 \mid i\in {\rm supp}(x)\}
$$
for each $i\in\{1,2,\hdots,n\}$. Clearly, in $\Sy_n$ any pair of transpositions with joint supports are contained in one of $\mathscr{F}_1,\mathscr{F}_2,\hdots,\mathscr{F}_n$. In addition, $\mathscr{F}_1\cup\cdots\cup\mathscr{F}_n$ contains all transpositions of $\Sy_n$ and each transposition of $\Sy_n$ appears exactly three times in $\mathscr{F}_0\cup\mathscr{F}_1\cup\cdots\cup\mathscr{F}_n$.

The following auxiliary graphs will play an important role in our proof of Theorem \ref{main theorem}.

\begin{dfn}\label{auxiliary graph}
{\em
Let $G$ be a spanning subgraph of $CT_n$. For each $i\in\{0,1,2,\hdots,n\}$ and each $x\in \Sy_n$, define $G_x^i$ to be the graph with vertex set $V(G_x^i) = \{yx\in \Sy_n \mid y\in \mathscr{F}_i\}$ such that for $u, z \in V(G_x^{i})$, $u$ and $z$ are adjacent if and only if $|\supp(\{x,u\})\cap \supp(\{x,z\})|= \delta_{i}$ and there exists a vertex $w$ with $w\neq x$ such that $(u,w,z)$ is a $2$-path in $G$, where $\delta_{0} = 0$ and $\delta_{i} = 1$ for $i \in \{1,2,\hdots,n\}$.
}
\end{dfn}

By the definition of $G_x^i$, it is clear that
\begin{align}
\sum\limits_{x\in{V(CT_n)}}\sum\limits_{i=0}^n v(G_x^i)&=3v\cdot{n\choose 2},
\label{VGx}
\end{align}
where as before $v=n!$ is the number of vertices of $CT_n$. Since $|V(G_x^i)\cap V(G_x^{j})|=1$ and $|E(G_x^0)\cap E(G_x^{i})|=0$ for any $i, j\in \{1,2,\hdots,n\}$ with $i \ne j$, we have $|E(G_x^i)\cap E(G_x^{j})|=0$ for any $i, j\in \{0,1,\hdots,n\}$ with $i \ne j$. Hence
\begin{align}
\sum\limits_{x\in{V(CT_n)}}\sum_{i=0}^{n}e(G_x^i)\geq\sum\limits_{w\in V(G)}{d_G(w)\choose 2},\label{EGx}
\end{align}
where the right-hand side gives the number of $2$-paths in $G$.

\begin{lem}\label{cycle in Gx}
Let $G$ be a spanning subgraph of $CT_n$. Let $l$ be an integer with $l\geq 3$. If there exists an $l$-cycle in $G_x^i$ for some $i\in \{0,1,\hdots,n\}$ and $x\in V(CT_n)$, then there exists a $2l$-cycle in $G$.

\end{lem}
\begin{proof}
Assume that $C=(u_1,u_2,\hdots,u_l,u_{l+1}=u_1)$ is an $l$-cycle in $G_x^i$. By the definition of $G_x^i,$ for each $j\in\{1,2,\hdots,l\}$ there exists $w_j\in V(CT_n)$ such that $(u_j,w_j,u_{j+1})$ is a $2$-path in $G$. By Corollary \ref{rem1}, we have $w_j=u_jx^{-1}u_{j+1}=u_{j+1}x^{-1}u_j$ if $i=0$, and $w_j\in\{u_jx^{-1}u_{j+1},u_{j+1}x^{-1}u_j\}$ if $i\in\{1,2,\ldots,n\}$.

We claim that $w_1,w_2,\hdots,w_l$ are pairwise distinct. Suppose to the contrary that $w_j=w_s$ with $j<s$. Since there are at most two cycles of length $4$ containing $(x, u_j,w_j)$, we have $s=j+1.$ If $i=0$, then $|\supp(\{x,u_j\})\cap \supp(\{x,u_{j+1}\})|=0$, which implies that $(xu_j^{-1})(u_{j+1}x^{-1})=(u_{j+1}x^{-1})(xu_j^{-1})=u_{j+1}u_j^{-1}$ and $w_j^{-1}w_{j+1}=u_{j+1}^{-1}xu_j^{-1}u_{j+1}x^{-1}u_{j+2}=u_j^{-1}u_{j+2}\neq \mathrm{id}$, a contradiction. If $i\in\{1,2,\ldots,n\}$, then $$\{u_{j}x^{-1}u_{j+1},u_{j+1}x^{-1}u_{j}\}\cap\{u_{j+1}x^{-1}u_{j+2},u_{j+2}x^{-1}u_{j+1}\}\neq \emptyset.$$ Assume that $u_j=(i,t_0)x$, $u_{j+1}=(i,t_1)x$ and $u_{j+2}=(i,t_2)x$, where $t_0, t_1, t_2$ are distinct elements of $\{1,2,\hdots,n\}\setminus \{i\}.$ Then $$\{(i,t_0,t_1)x,(i,t_1,t_0)x\}\cap\{(i,t_1,t_2)x,(i,t_2,t_1)x\}\neq \emptyset,$$ which is impossible.

Since $CT_n$ is a bipartite graph, we have $\{u_1,u_2,\hdots,u_l\}\cap\{w_1,w_2,\hdots,w_l\}=\emptyset.$ Since $w_1,w_2,\hdots,w_l$ are pairwise distinct and the $2$-path $(u_j,w_j,u_{j+1})$ is in $G$ for $j\in\{1,2,\hdots,l\}$, it follows that $(u_1,w_1,u_2,w_2,\hdots,u_l,w_l,u_{l+1}=u_1)$ is a $2l$-cycle in $G$.
\end{proof}

\section{Proof of the main result when $l\geq4$}
\label{large}

We prove parts (i) and (ii) of Theorem \ref{main theorem} in this section.

\subsection{$4k$-cycle-free subgraphs of $CT_n$}

\noindent \textit{Proof of Theorem \ref{main theorem} (i).}
Suppose $G$ is a $C_{4k}$-free spanning subgraph of $CT_n$ with maximum number of edges, where $k\geq 2$. Then $d_G(w)\geq1$ for each $w\in V(G).$
Since $G$ is $C_{4k}$-free, by Lemma \ref{cycle in Gx}, $G_x^i$ is $C_{2k}$-free for any $x\in V(CT_n)$ and $i\in \{0,1,\hdots,n\}.$ Thus from the main theorem in \cite{Bondy1974} by Bondy and Simonovits it follows that $G_x^i$ has at most $c_k(v(G_x^i))^{1+\frac{1}{k}}$ edges, where $c_k$ is a positive constant relying on $k$ only. Therefore, we have
\begin{align}\label{upperbound}
\sum\limits_{x\in{V(CT_n)}}\sum_{i=0}^{n}e(G_x^i)\leq c_k v\cdot \left({n\choose 2}^{1+\frac{1}{k}}+n(n-1)^{1+\frac{1}{k}}\right)\leq c_k' v\cdot {n\choose 2}^{1+\frac{1}{k}} .
\end{align}
On the other hand, by (\ref{EGx}) and the Cauchy-Schwarz inequality, we have
\begin{align}\label{lowerbound}
\sum\limits_{x\in{V(CT_n)}}\sum_{i=0}^{n}e(G_x^i) & \geq \sum\limits_{w\in V(G)}{d_G(w)\choose 2} \notag \\
& = \frac{1}{2}\sum\limits_{w\in V(G)}d_G(w)^2-\frac{1}{2}\sum\limits_{w\in V(G)}d_G(w)\notag\\
& \geq \frac{1}{2v}\left(\sum\limits_{w\in V(G)}d_G(w)\right)^2-\frac{1}{2}\sum\limits_{w\in V(G)}d_G(w).
\end{align}

Since $\sum_{w\in V(G)}d_G(w)=2e(G)$, it follows from (\ref{upperbound}) and (\ref{lowerbound}) that
\begin{align*}
\frac{2e(G)^2}{v}-e(G)\leq c_k' v\cdot{n\choose 2}^{1+\frac{1}{k}},
\end{align*}
or equivalently,
\begin{align*}
e(G)^2\leq \frac{1}{2}c_k' v^2\cdot{n\choose 2}^{1+\frac{1}{k}}+\frac{e(G)v}{2}.
\end{align*}
Set $\pi=e(G)/e(CT_n)$. Observe that $0 < \pi < 1$. Since $e(CT_n)=\frac{v}{2}{n\choose 2}$, the inequality above yields
\begin{align*}
\pi^2&\leq 2c_k'\cdot{n\choose 2}^{-1+\frac{1}{k}}+\pi\cdot{n\choose 2}^{-1}.
\end{align*}
So there exists a constant $c$ depending on $k$ such that
\begin{align*}
\pi&\leq cn^{-1+\frac{1}{k}}.
\end{align*}
Therefore, we have
$$
\ex(CT_n, C_{4k}) = e(G) = \pi e(CT_n) \leq cn^{-1+\frac{1}{k}} e(CT_n),
$$
as desired in part (i) of Theorem~\ref{main theorem}. \qed

\subsection{$(4k+2)$-cycle-free subgraphs of $CT_n$}

In this subsection we assume that $G$ is a $C_{4k+2}$-free spanning subgraph of $CT_n$ and $a$ and $b$ are integers with $a,b\geq2$ such that $4a+4b=4k+4$, where $k\geq 2$. Note that a cycle of length $4a$ in $G$ can not intersect a cycle of length $4b$ in $G$ at a single edge, for otherwise their union would contain a cycle of length $4k+2$. In what follows we will give an upper bound as well as a lower bound on the number of $4a$-cycles in $G$. These bounds will be used in the proof of part (ii) of Theorem~\ref{main theorem} at the end of this subsection.

\begin{lem}\label{lemDC1}
For any $2l$-cycle $C$ in $CT_n$, where $l\geq2$, we have $|\supp(C)|\leq 2l.$
\end{lem}
\begin{proof}
Let $C=(u_0,u_1,u_2,\ldots, u_{2l}=u_0)$ be a $2l$-cycle in $CT_n$. Set $w_i = u_iu_{i-1}^{-1}$ for $i\in\{1,2,\ldots,2l\}.$ Then $\supp(\{u_{i-1},u_i\})=\supp(w_i)$ and $\supp(C)=\cup_{i=1}^{2l}\supp(w_i)$. Observe that $w_{2l}w_{2l-1}\cdots w_2w_1 = \mathrm{id}$. So for any $x\in \supp(C)$ there exist distinct $i,j\in\{1,2,\ldots,2l\}$ such that $x\in \supp(w_i)\cap \supp(w_j)$. Hence $|\supp(C)|\leq 2l.$
\end{proof}

\begin{lem}\label{lemDC2}
Let $C$ and $C'$ be cycles of lengths $4a$ and $4b$ in $G$, respectively. If $C$ and $C' $ have at least one common edge, then $|\supp(C)\cap \supp(C')|\geq 3$.
\end{lem}
\begin{proof}
Suppose $\{u_1,u_2\}$ is a common edge of $C$ and $C'$. Since $G$ is a $(4a+4b-2)$-cycle-free subgraph of $CT_n$, there exists a vertex $u_3$ of $G$ such that $u_3\in (V(C)\cap V(C')) \setminus \{u_1,u_2\}$. Since $\supp(u_3u_1^{-1})\subseteq \supp(C)\cap \supp(C^\prime)$ and $u_3\neq u_2$, we have
$$
\supp(\{u_1,u_2\})\neq \supp(u_3u_1^{-1}).
$$
This together with $\supp(\{u_1,u_2\})\subseteq \supp(C)\cap \supp(C^\prime)$ implies that $|\supp(C)\cap \supp(C^\prime)|\geq 3$.
\end{proof}

For any graphs $H$ and $L$, define $N(H, L)$ to be the number of subgraphs of $H$ which are isomorphic to $L$.

\begin{lem}\label{upC4a}
We have
$$N(G,C_{4a})= O(n^{4a-3})e(G)+O(vn^{4a-1+\frac{1}{b}}).$$
Moreover, if $a=b$, then $N(G,C_{4a})= O(n^{4a-3})e(G)$.
\end{lem}
\begin{proof}
Denote by $\mathscr{C}$ the set of cycles of length $4a$ in $G$ and $\mathscr{C}_e$ the set of cycles in $\mathscr{C}$ containing a given edge $e$. Note that $|\mathscr{C}|=N(G,C_{4a})$. Let $E = \cup_{C\in\mathscr{C}}E(C)$. Let $E_1$ be the set of edges in $E$ that are contained in a cycle of length $4b$ in $G$, and let $E_2:=E \setminus E_1$. Then $E = E_1 \cup E_2$ and
\begin{align}\label{4aNe1e2}
4aN(G,C_{4a})=\sum_{e_1\in E_1 }|\mathscr{C}_{e_1}|+\sum_{e_2\in E_2 }|\mathscr{C}_{e_2}|.
\end{align}

Assume that $e=\{u_1,u_{4a}\}$. Observe that for any $4a$-cycle $(u_1,u_2,\hdots,u_{4a},u_1)$, there is a unique sequence  $(A_1,A_2,\hdots,A_{4a-1})$ of length $4a-1$ such that $A_i=\supp(\{u_i,u_{i+1}\})$ for any $i\in\{1,2,\hdots,4a-1\}$. For each $B\in\{\supp(C^*) \mid C^*\in \mathscr{C}_e\}$, there are ${|B|\choose 2}^{4a-1}$ sequences $(A_1,A_2,\hdots,A_{4a-1})$ of length $4a-1$ such that $A_i\subseteq B$ and $|A_i|=2$ for each $i\in\{1,2,\hdots,4a-1\}$, and hence there are at most ${|B|\choose 2}^{4a-1}$ $4a$-cycles $C$ containing $e$ such that $\supp(C)=B$.

For each $e_1\in E_1$ (if $E_1\neq\emptyset$), let $C'$ be a fixed $4b$-cycle with $e_1\in E(C').$ For any $4a$-cycle $C^*\in \mathscr{C}_{e_1},$ we have $\supp(e_1)\subseteq \supp(C^*)$ and $|\supp(C^*)\cap \supp(C')|\geq 3$ by Lemma~\ref{lemDC2}. Hence, by Lemma~\ref{lemDC1}, we have
\begin{align*}
|\{\supp(C^*)\mid C^*\in\mathscr{C}_{e_1}\}|\leq\sum_{i=1}^{4a-2}{|\supp(C')|-2\choose i}\sum_{j=0}^{4a-2-i}{n-|\supp(C')|\choose j},
\end{align*}
which implies
\begin{align}\label{Ce1}
|\mathscr{C}_{e_1}| & \leq\sum_{i=1}^{4a-2}{|\supp(C^\prime)|-2\choose i}\sum_{j=0}^{4a-2-i}{n-|\supp(C')|\choose j}{i+j+2\choose 2}^{4a-1} \notag \\
& =O(n^{4a-3}).
\end{align}
For each $e_2\in E_2$ (if $E_2\neq\emptyset$),  by Lemma~\ref{lemDC1} again, we have
\begin{align*}
|\{\supp(C^*)\mid C^*\in\mathscr{C}_{e_2}\}|\leq\sum_{i=0}^{4a-2}{n-2\choose i},
\end{align*}
which implies
\begin{align}\label{Ce2}
|\mathscr{C}_{e_2}| & \leq\sum_{i=0}^{4a-2}{n-2\choose i}{i+2\choose 2}^{4a-1} \notag \\
& = O(n^{4a-2}).
\end{align}

Note that $|E_1|\leq e(G)$. Note also that $|E_2|\leq {\rm ex}(CT_n,C_{4b})$ as the subgraph induced by $E_2$ is $C_{4b}$-free. Using part (i) of Theorem \ref{main theorem} (which has been proved already), we have $|E_2|\leq cn^{-1+1/b}e(CT_n)$ for some positive constant $c$. Combining (\ref{4aNe1e2}), (\ref{Ce1}) and (\ref{Ce2}), we obtain
\begin{align*}
N(G,C_{4a}) & \leq \frac{1}{4a}\left(\sum\limits_{e\in E_1}O(n^{4a-3})+\sum\limits_{e\in E_2}O(n^{4a-2})\right) \\
& \leq O(n^{4a-3})e(G)+O(vn^{4a-1+\frac{1}{b}}).
\end{align*}
In particular, if $a=b$, then $|E_2|=0$ and hence
\begin{align*}
N(G,C_{4a})\leq\frac{1}{4a}\sum\limits_{e\in E_1}O(n^{4a-3})\leq O(n^{4a-3})e(G).
\end{align*}
This completes the proof.
\end{proof}

\begin{prop}\label{Erdosthm}{\rm(Erd\H{o}s and Simonovits~\cite{Erdos1982})}  ~Let $L$ be a bipartite graph, where there exist vertices $x$ and $y$ such that $L\setminus\{x,y\}$ is a tree. Then there exist constants $c_1,\ c_2 > 0$ such that if $H$ is a graph containing more than $c_1v(H)^{\frac{3}{2}}$ edges, then
$$
N(H,L)\geq c_2 \frac{e(H)^{e(L)}}{v(H)^{2e(L)-v(L)}}.
$$
\end{prop}

With the help of this proposition and the auxiliary graphs $G_x^i$ as defined in Definition \ref{auxiliary graph}, we now prove a lower bound on $N(G,C_{4a})$.

\begin{lem}\label{lowC4a}
We have
$$
N(G,C_{4a})\geq c v\frac {d^{4a}}{n^{4a}}-O(vn^{2a})
$$
for some positive constant $c$ depending on $a$, where $d=2e(G)/v$.
\end{lem}
\begin{proof}
By Lemma~\ref{cycle in Gx}, we have
\begin{equation}\label{NGNC}
N(G,C_{4a})\geq\sum_{x\in V(CT_n)}\sum_{i=0}^{n} N(G_x^i,C_{2a}).
\end{equation}
Setting $L=C_{2a}$ in Proposition~\ref{Erdosthm}, there exist two positive constants $c_1$ and $c_2$ such that
\begin{align*}
N(G_x^i,C_{2a})
\geq c_2\left(\frac{e(G_x^i)^{2a}}{v(G_x^i)^{2a}}-\frac{(c_1v(G_x^i)^{3/2})^{2a}}{v(G_x^i)^{2a}}\right).
\end{align*}
Combining this with \eqref{NGNC}, we obtain
\begin{align*}
  N(G,C_{4a})&\geq \sum\limits_{x\in
  V(CT_n)}\sum_{i=0}^{n} c_2\left(\frac{e(G_x^i)^{2a}}{v(G_x^i)^{2a}}-\frac{(c_1v(G_x^i)^{3/2})^{2a}}{v(G_x^i)^{2a}}\right)\\
  &\geq\sum\limits_{x\in
  V(CT_n)}\left(\frac{c_2 e(G_x^0)^{2a}}{{n\choose 2}^{2a}}+\sum_{i=1}^{n}\frac{c_2 e(G_x^i)^{2a}}{(n-1)^{2a}}\right)-\sum\limits_{x\in
  V(CT_n)}\sum_{i=0}^{n}c_1^{2a}v(G_x^i)^{a}.
\end{align*}
By H\"{o}lder's inequality, we then have
\begin{align*}
N(G,C_{4a})&\geq c_2\sum_{x\in V(CT_n)}\left(\frac{e(G_x^0)^{2a}}{{n\choose 2}^{2a}}+\frac{\left(\sum_{i=1}^{n}e(G_x^i)\right)^{2a}}{n^{2a-1}(n-1)^{2a}}\right)-O(v n^{2a})\\
  &\geq \frac{c_2}{n^{4a}}\sum_{x\in V(CT_n)}\left(e(G_x^0)^{2a}+\left(\sum_{i=1}^{n}e(G_x^i)\right)^{2a}\right)-O(v n^{2a})\\
  &\geq \frac{c_a}{n^{4a}}\sum_{x\in V(CT_n)}\left(\sum_{i=0}^{n}e(G_x^i)\right)^{2a}-O(v n^{2a})\\
  &\geq \frac{c_a v}{n^{4a}}\left(\sum_{x\in V(CT_n)}\sum_{i=0}^{n}\frac{e(G_x^i)}{v}\right)^{2a}-O(v n^{2a})\\
  &\geq \frac{c_a v}{n^{4a}}\left(\sum_{w\in V(G)}\frac{{d_G(w)\choose 2}}{v}\right)^{2a}-O(v n^{2a}),
\end{align*}
where $c_a$ is a positive constant depending on $a$ and inequality \eqref{EGx} is used in the last step. Setting $d=2e(G)/v$ and applying H\"{o}lder's inequality again, we obtain
\begin{align*}
N(G,C_{4a})
&\geq \frac{c_a v}{n^{4a}}{\sum_{w\in V(G)}\frac{d_G(w)}{v}\choose 2}^{2a}-O(v n^{2a})\\
&=\frac{c_a v}{n^{4a}}{d\choose 2}^{2a}-O(v n^{2a})\\
&\geq cv\frac{d^{4a}}{n^{4a}}-O(v n^{2a})
\end{align*}
for some positive constant $c$ depending on $a$. This completes the proof.
\end{proof}

\medskip

\noindent \textit{Proof of Theorem \ref{main theorem} (ii).}
Suppose $G$ is a $C_{4k+2}$-free spanning subgraph of $CT_n$ with maximum number of edges. Then $\ex(CT_n, C_{2l}) = e(G)$, where $l = 2k+1$. Set $d=2e(G)/v$. Combining Lemma \ref{upC4a} and Lemma \ref{lowC4a}, we have
\begin{align*}
c v\frac {d^{4a}}{n^{4a}}&\leq O(n^{4a-3})e(G)+O(vn^{4a-1+\frac{1}{b}})+O(vn^{2a}),\\
d^{4a}&\leq O(n^{8a-3})d+O(n^{8a-1+\frac{1}{b}})+O(n^{6a}).
\end{align*}
Hence $d = \max \left\{O(n^{2-\frac{1}{4a-1}}),\ O(n^{2-\frac{1-\frac{1}{b}}{4a}})\right\}$. This bound is minimized when $a=2$ and $b=k-1$, and this choice of $(a, b)$ yields $d =O(n^{2-\frac{1}{8}+\frac{1}{8(k-1)}})$. Since $e(G) = vd/2$ and $e(CT_n)=\frac{v}{2}{n\choose 2}$, it follows that
\begin{align}
\label{Oin1}
e(G) & = O(vn^{2-\frac{1}{8}+\frac{1}{8(k-1)}}) \notag \\
& = O(n^{-\frac{1}{8}+\frac{1}{8(k-1)}})e(CT_n) \notag \\
& = O(n^{-\frac{1}{8}+\frac{1}{4(l-3)}})e(CT_n).
\end{align}

Consider the case when $a=b=(k+1)/2$ with $k$ odd. By Lemmas \ref{upC4a} and \ref{lowC4a}, we have
\begin{align*}
d^{4a}\leq O(n^{8a-3})d+ O(n^{6a}),
\end{align*}
which yields
\begin{align}\label{Oin2}
e(G) & = O(n^{2-\frac{1}{4a-1}}) \notag \\
& = O(vn^{2-\frac{1}{2k+1}}) \notag \\
& = O(n^{-\frac{1}{2k+1}})e(CT_n) \notag \\
& = O(n^{-\frac{1}{l}})e(CT_n).
\end{align}
Observe that when $k$ is odd we have $n^{-\frac{1}{2k+1}}\leq n^{-\frac{1}{8}+\frac{1}{8(k-1)}}$ if and only if $0 < k < 4.9$. So $(\ref{Oin2})$ is a better bound than $(\ref{Oin1})$ when $k=3$. Therefore, $e(G) = O(n^{-\frac{1}{l}})e(CT_n)$ when $l = 7$. This competes the proof. \qed

So far we have completed the proof of Theorem \ref{main theorem}\ (i) and (ii). These results imply that $\ex(CT_n,C_{2l})=o(e(CT_n))$ for any fixed positive integer $l\geq4$. Thus, for any $t \ge 1$ and $l \ge 4$, there exists a positive integer $n(t,l)$ such that for any $n>n(t,l)$ and any edge-coloring of $CT_n$ with $t$ colors, $CT_n$ contains a monochromatic copy of $C_{2l}$, as claimed in Corollary \ref{cor1}.

\begin{rem}
{\em
The \emph{theta graph} $\Theta_{i,j,k}$ is the graph with $i+j+k-1$ vertices which consists of three internally vertex-disjoint paths between the same pair of vertices with lengths $i$, $j$ and $k$, respectively. As a by-product of the proof of Theorem \ref{main theorem}\ (i) and (ii), we obtain that
$$
{\rm ex}(CT_n, \Theta_{4a-1,1,4b-1}) = o(e(CT_n))
$$
for any $a,b\geq2$.}
\end{rem}

\section{Proof of the main result when $l=2,3$}
\label{small}

In this section, $\mathscr{C}_4$ denotes the set of $4$-cycles in $CT_n$, and for each $e\in E(CT_n)$, $(\mathscr{C}_4)_e$ denotes the set of $4$-cycles in $CT_n$ containing $e$. Suppose $G$ is a $2l$-cycle-free spanning subgraph of $CT_n$ with maximum number of edges. For any subgraphs $H$ and $L$ of $CT_n$, let $G\cap H$ be the graph with vertex set $V(G)\cap V(H)$ and edge set $E(G)\cap E(H)$.

Note that for any $4$-cycle $H \in \mathscr{C}_4$, $G\cap H$ is isomorphic to one of the six graphs in Figure \ref{fig1}. Denote by $\chi_0, \chi_1, \chi_2^{1}, \chi_2^{2}, \chi_3, \chi_4$ the ratio of the number of $4$-cycles $H$ with $G\cap H$ isomorphic to the graphs (1)--(6) in Figure \ref{fig1} to the total number of $4$-cycles in $CT_n$, respectively. Of course we have
\begin{equation}\label{chi1}
\chi_0+\chi_1+\chi_2^{1}+\chi_2^{2}+\chi_3+\chi_4=1.
\end{equation}
By double counting the cardinality of $\{(e, H)\mid H\in\mathscr{C}_4,\ e\in  E(G\cap H)\}$, we obtain
\begin{align*}
\sum_{H\in\mathscr{C}_4}e(G\cap H)=\sum_{e\in E(G)}|(\mathscr{C}_4)_e|,
\end{align*}
which by Corollary~\ref{4cycle_2} (ii) implies
\begin{align*}
\left(\chi_1+2(\chi_2^{1}+\chi_2^{2})+3\chi_3+4\chi_4\right)\cdot n(C_4)&=e(G)\cdot\frac{1}{2}(n-2)(n+1),
\end{align*}
where as before $n(C_4)$ is the number of $4$-cycles in $CT_n$. Set $\pi=e(G)/e(CT_n)$. By Corollary~\ref{4cycle_2} (iii), we have
\begin{equation}
\label{4pi}
\chi_1+2(\chi_2^{1}+\chi_2^{2})+3\chi_3+4\chi_4 = 4\pi.
\end{equation}

\medskip

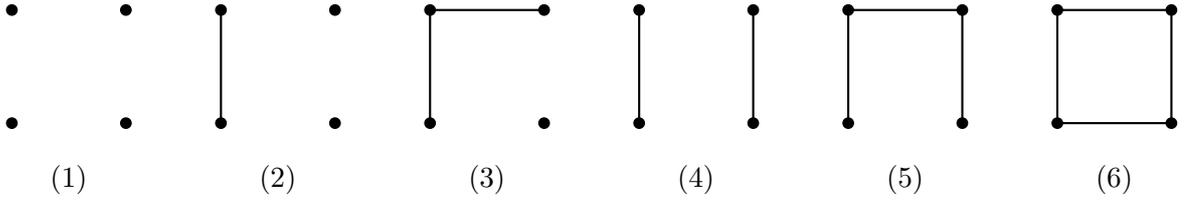
\begin{figure}
\begin{center}
 \begin{tikzpicture}[thick]
  \draw[black,thick] (2.75,0) -- (2.75,1.5);
  \draw[black,thick] (5.5,0) -- (5.5,1.5) -- (7,1.5);
  \draw[black,thick] (8.25,0) -- (8.25,1.5); \draw[black,thick] (9.75,1.5) -- (9.75,0);
  \draw[black,thick] (11,0) -- (11,1.5) -- (12.5,1.5) -- (12.5,0);
  \draw[black,thick] (13.75,0) -- (13.75,1.5) -- (15.25,1.5) -- (15.25,0) -- cycle;
  \filldraw [black] (0,0)circle (1.8pt) (0,1.5)circle (1.8pt) (1.5,1.5)circle (1.8pt) (1.5,0)circle (1.8pt);
  \filldraw [black] (2.75,0)circle (1.8pt) (2.75,1.5)circle (1.8pt) (4.25,1.5)circle (1.8pt) (4.25,0)circle (1.8pt);
  \filldraw [black] (5.5,0)circle (1.8pt) (5.5,1.5)circle (1.8pt) (7,1.5)circle (1.8pt) (7,0)circle (1.8pt);
  \filldraw [black] (8.25,0)circle (1.8pt) (8.25,1.5)circle (1.8pt) (9.75,1.5)circle (1.8pt) (9.75,0)circle (1.8pt);
  \filldraw [black] (11,0)circle (1.8pt) (11,1.5)circle (1.8pt) (12.5,1.5)circle (1.8pt) (12.5,0)circle (1.8pt);
  \filldraw [black] (13.75,0)circle (1.8pt) (13.75,1.5)circle (1.8pt) (15.25,1.5)circle (1.8pt) (15.25,0)circle (1.8pt);
  \node at (0.75,-0.75) {(1)};\node at (3.5,-0.75) {(2)};\node at (6.25,-0.75) {(3)};
  \node at (9,-0.75) {(4)};\node at (11.75,-0.75) {(5)};\node at (14.5,-0.75) {(6)};
 \end{tikzpicture}
\end{center}
\caption{Possibilities for $G \cap H$ when $H \in \mathscr{C}_4$.}
\label{fig1}
\end{figure}

\noindent \textit{Proof of Theorem \ref{main theorem} (iv).}
Suppose $G$ is a $C_4$-free spanning subgraph of $CT_n$ with maximum number of edges. Then $d_G(w)\geq1$ for any $w\in V(G)$ and $\chi_4 = 0$ as $G$ is $C_4$-free. Hence, by \eqref{chi1} and \eqref{4pi}, we have
$$
\pi=\frac{1}{4}\left(\chi_1+2(\chi_2^{1}+\chi_2^{2})+3\chi_3\right) \leq
\frac{3}{4}\left(\chi_0 + \chi_1+\chi_2^{1}+\chi_2^{2}+\chi_3\right) = \frac{3}{4}.
$$
Thus $\ex(CT_n,C_4) = e(G) = \pi e(CT_n) \leq \frac{3}{4}e(CT_n)$ as desired in part (iv) of Theorem \ref{main theorem}.
\qed

\medskip

\noindent \textit{Proof of Theorem \ref{main theorem} (iii).}
Suppose $G$ is a $C_6$-free spanning subgraph of $CT_n$ with maximum number of edges. For each $i \in \{0, 1, 2, \ldots, n\}$ and each $x\in V(CT_n)$, let $H_x^i$ be the subgraph of $G_x^i$ (see Definition \ref{auxiliary graph}) induced by the subset $\{u\in V(G_x^i)\mid \{u,x\}\notin E(G)\}$ of $V(G_x^i)$.
Then
$$
\sum\limits_{x\in{V(CT_n)}}\sum\limits_{i=0}^n v(H_x^i) = 3\sum\limits_{x\in V(CT_n)}\left({n\choose 2}-d_G(x)\right).
$$
Since $|E(H_x^i)\cap E(H_x^{j})|=0$ for distinct $i, j \in \{0,1,\hdots,n\}$, we have
\begin{align}
\sum\limits_{x\in{V(CT_n)}}\sum_{i=0}^{n}e(H_x^i)+(4\chi_4+2\chi_3)\cdot n(C_4)\geq \sum\limits_{w\in V(G)}{d_G(w)\choose 2}.\label{aue}
\end{align}

We claim that for any $e\in E(CT_n)$ there are at most two $4$-cycles $H$ in $\mathscr{C}_4$ containing $e$ such that $(H\cap G)-e$ is isomorphic to the graph $(5)$ in Figure \ref{fig1}. Suppose to the contrary that there exist three such $4$-cycles in $\mathscr{C}_4$, say, $C_1,C_2$ and $C_3$. Suppose $e=\{u,z\}$. Since $G$ is $C_6$-free, we have $(V(C_i)\setminus \{u,z\})\cap (V(C_j)\setminus \{u,z\})\neq \emptyset$ for any distinct $i,j\in\{1,2,3\}$. Setting $C_1=(u,z,x_1,x_2,u)$ and $C_2=(u,z,x_1,x_3,u)$. If $V(C_3)=\{u,z,x_1,x_4\}$, then there are three $4$-cycles containing the $2$-path $(u,z,x_1)$, which contradicts Corollary \ref{rem1}. If $V(C_3)=\{u,z,x_2,x_3\}$, then there exists a triangle in $G$, a contradiction. This proves our claim. By double counting the number of pairs $(e,H)$ with $e\in E(CT_n)$ and $H\in\mathscr{C}_4$ such that $(G\cap H) - e$ is isomorphic to the graph $(5)$ in Figure \ref{fig1}, we obtain $2e(CT_n)\geq (\chi_3 + 4\chi_4)\cdot n(C_4)$. This together with Corollary \ref{4cycle_2} (iii) implies
$\chi_3 + 4\chi_4 \le 2e(CT_n)/n(C_4) = 16/(n-2)(n+1)$. Therefore,
\begin{equation}\label{p1}
2\chi_3 + 4\chi_4 = o(1).
\end{equation}

Since $G$ is $C_6$-free and $H_x^i$ is a subgraph of $G_x^i$, by Lemma~\ref{cycle in Gx}, $H_x^i$ contains no $3$-cycles for any $x\in V(CT_n)$ and $i\in\{0,1,\hdots,n\}$. So by Mantel's theorem \cite{Mantel1907} we have $e(H_x^0)\leq\left({n\choose 2}-d_G(x)\right)^2/4$ and $e(H_x^i)\leq(n-1)^2/4$ for $i\in\{1,2,\hdots,n\}$. Since $|E(H_x^i)\cap E(H_x^{j})|=0$ for distinct $i, j\in \{0,1,\hdots,n\}$, we have
\begin{align*}
\sum\limits_{i=0}^{n}e(H_x^i) &\leq \frac{1}{4} \left({n\choose 2}-d_G(x)\right)^2 + \frac{1}{4} n(n-1)^2\\
&=\frac{1}{4}\left({n\choose 2}^{2}+n(n-1)^2-2 {n\choose 2} d_G(x)+d_G(x)^2\right).
\end{align*}
Since $\sum_{x\in V(G)}d_G(x)=2e(G)$, it follows that
\begin{align}
\sum\limits_{x\in{V(CT_n)}}\sum_{i=0}^{n}e(H_x^i)&\leq \frac{v}{4}{n\choose 2}^2-{n\choose 2}e(G)+\frac{vn(n-1)^2}{4}+\frac{1}{4}\sum\limits_{x\in V(CT_n)}d_G(x)^2.\label{eq2}
\end{align}
One the other hand, by (\ref{aue}) and (\ref{p1}), we have
\begin{align}
\sum\limits_{x\in{V(CT_n)}}\sum_{i=0}^{n}e(H_x^i)&\geq \sum\limits_{w\in V(G)}{d_G(w)\choose 2}-o(n(C_4))\label{eq1},\\
&\geq \frac{1}{2}\sum\limits_{w\in V(G)}d_G(w)^2-e(G)-o(n(C_4)).\notag
\end{align}

Combining (\ref{eq2}) with (\ref{eq1}), we have
\begin{align*}
\frac{v}{4}{n\choose 2}^2-{n\choose 2}e(G)+\frac{vn(n-1)^2}{4} & \geq \frac{1}{4}\sum\limits_{w\in V(G)}d_G(w)^2-e(G)-o(n(C_4)) \\
& \geq \frac{1}{4v} \left(\sum\limits_{w\in V(G)}d_G(w)\right)^2-e(G)-o(n(C_4)) \\
& = \frac{e(G)^2}{v}-e(G)-o(n(C_4)).
\end{align*}
Dividing both sides by $\frac{v}{4}{n\choose 2}^2$, we then obtain
$$
1-\frac{2e(G)}{e(CT_n)}+\frac{4}{n}-\frac{e(G)^2}{e(CT_n)^2}+\frac{2e(G)}{e(CT_n){n\choose 2}}+o(1)\geq 0.
$$
Recall that $\pi=e(G)/e(CT_n)$. Since $0 < \pi < 1$, $\frac{4}{n}=o(1)$ and $2\pi/{n\choose 2}=o(1)$, we have $1-2\pi-\pi^2+o(1)\geq 0$, which implies $\pi\leq \sqrt{2}-1+o(1)$. Therefore, we have $e(G) = \pi e(CT_n) \leq (\sqrt{2}-1+o(1))e(CT_n)$ as desired in part (iii) of Theorem \ref{main theorem}.
\qed

\vskip0.1in
\noindent\textsc{Acknowledgement.} This research was supported by NSFC (11671043).

\end{document}